\documentclass[11pt]{article}
\usepackage[latin5]{inputenc}
\usepackage{amsfonts}
\usepackage{amssymb}
\usepackage{amsmath}
\usepackage{cite}
\usepackage{bbm}
\usepackage{graphicx}
\usepackage{pstricks,pst-math,pst-infixplot}
\usepackage[dotinlabels]{titletoc}
\usepackage{psfrag}

\newcounter{theorem}
\newtheorem{theorem}{Theorem}

\newtheorem{corollary}{Corollary}
\newtheorem{definition}{Definition}

\newtheorem{lemma}{Lemma}
\newtheorem{proposition}{Proposition}

\newenvironment{proof}[1][Proof]{\textbf{#1.} }{\rule{0.5em}{0.5em}}

\begin{document}

\title{Tight Span of Path Connected Subsets of the Manhattan Plane}
\author{Mehmet KILIÇ and Şahin KOÇAK}
\date{ }
\maketitle

\begin{abstract}
We show that the tight span of a path connected subset of the
Manhattan plane can be constructed in a very simple way by
"hatching" the subset in horizontal and vertical directions and
then taking the closure of the resulting set.
\end{abstract}

\section{Introduction}

The hyperconvex hull of a metric space defined by Aronszajn and Panitchpakdi~\cite{aro}, the injective hull of a metric space defined by Isbell~\cite{isb}
and the tight span of a metric space defined by Dress~\cite{dre}, all turned out to be isometric to each other and they play an important role in metric geometry. Especially in the form of tight span, they proved to be instrumental in phylogenetic modelling. It is however a nontrivial task to construct the tight span of a given concrete metric space. So far as we know, there are no general constructions going beyond finite metric spaces. The definition itself uses a huge ambient space to embed the tight span
and it seems desirable to have examples of infinite metric spaces with a low dimensional tight span description.

Eppstein~\cite{epp} considers the subsets of the Manhattan plane
and gives a tight span construction, which works under certain
conditions.

In this note we consider any path connected subsets of the Manhattan plane and give a surprisingly simple description of the tight span using a procedure of ``hatching'' in the plane (see Theorem~\ref{35}).

We now want to recall briefly the definitions of the relevant notions. A metric space $(X,d)$ is called hyperconvex if for any collection $(x_i)_{i\in I}$ of points in $X$ and any collection $(r_i)_{i\in I}$ of nonnegative real numbers satisfying $d(x_i,x_j)\leq r_i+r_j$ for all $i,j\in I$, the intersection of closed balls around $x_i$ with radius $r_i$ is nonempty \cite{aro}.

A metric space $(X,d)$ is called injective, in the wording of Isbell, if every mapping which increases no distance from a subspace of any metric space $(Y,d')$ to $(X,d)$ can be extended, increasing no distance, over $(Y,d')$ \cite{isb}.

A metric space is hyperconvex if and only if it is injective and every metric space can be embedded into a minimal hyperconvex (injective) space called the hyperconvex (injective) envelope of the given metric space (unique up to isometry). We note the every useful property that a metric space $X$ is hyperconvex (injective) iff any isometric embedding $f:X\rightarrow X\cup\{y\}$, where $y\notin X$, has a nonexpansive retraction \cite{esp}.

A hyperconvex metric space is complete \cite{esp} and strictly
intrinsic \cite{kil} in the sense that between any two points
there is a geodesic, i.e. a path with length realising the
distance between these points (see for these notions \cite{bur}
and \cite{pap}).

The tight span of a metric space is nothing else than the injective envelope of the metric space, rediscovered by Dress \cite{dre}. The following can be taken as an explicit definition of tight span: Let $(X,d)$ be any metric space and consider the set of pointwise minimal functions $f:X\rightarrow \mathbb{R}^{\geq0}$ satisfying the property: \[f(x)+f(y)\geq d(x,y)\] for all $x,y\in X$. The tight span $T(X)$ of $X$ is then this set of functions with the supremum metric: \[d_{\infty}(f,g)=\sup_{x\in X}|f(x)-g(x)|.\]

We have previously shown that for any nonempty subset $A\subseteq\mathbb{R}_1^2$ of the Manhattan plane, any closed, geodesically convex (i.e. any two points of $A$ connected with a geodesic lying in $A$) and minimal (with regard these properties) subset $B\subseteq\mathbb{R}_1^2$ containing $A$ is isometric to the tight span $T(A)$ \cite{kil}.

In the present note we explicitly construct the tight span of a path connected subset $A$ of the Manhattan plane by using a simple procedure of ``hatching'' (see below for the definition of hatching). The tight span of $A$ will be the closure of the result of hatching the set $A$ successively in both axis-directions (see Theorem~\ref{35}). For a compact and path-connected set $A$, the tight span is simply the double-hatched set $A$.

\section{Hatching the Subsets of The Manhattan Plane}

We define the following two operations $L_x$ and $L_y$ on subsets of $\mathbb{R}_1^2=(\mathbb{R}^2,d_1)$ with $d_1((x_1,y_1),(x_2,y_2))=|x_1-x_2|+|y_1-y_2|$, called the Manhattan plane or the taxi plane. Informally $L_x(A)$ will be the result of hatching the set $A$ in the $x$ direction (similarly $L_y(A)$ in the $y$ direction) of the set $A\subseteq \mathbb{R}_1^2$. See, for example, Figure~\ref{3f2} for the action of $L_x$ and $L_y$ on the set  $A=\{e^{it}|\ t\in[0,\frac{3\pi}{2}]\}$

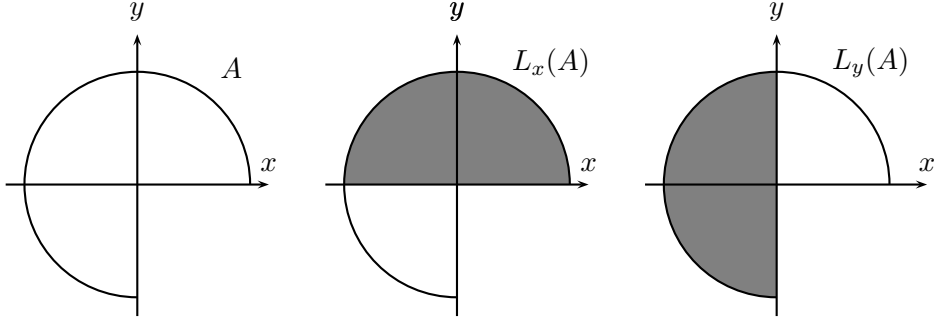
\begin{figure}[h]
\begin{center}
\begin{pspicture}(-5.25,-6.25)(7.5,-2)
\uput[u](-1.75,-4.5){$x$} \uput[u](-3.5,-2.5){$y$}
\psarc(-3.5,-4.5){1.5}{0}{270}
\uput[u](-2.25,-3.25){$A$}
\psline{->}(-5.25,-4.5)(-1.75,-4.5)
\psline{->}(-3.5,-6.25)(-3.5,-2.5) \uput[u](0.75,-2.5){$y$}
\uput[u](2.5,-4.5){$x$} \uput[u](0.75,-2.5){$y$}
\psarc(0.75,-4.5){1.5}{180}{270}
\psarc*[linecolor=gray](0.75,-4.5){1.5}{0}{180}
\psarc(0.75,-4.5){1.5}{0}{180} \uput[u](2,-3.25){$L_x(A)$}
\uput[u](7,-4.5){$x$} \psline{->}(-1,-4.5)(2.5,-4.5)
\psline{->}(0.75,-6.25)(0.75,-2.5) \uput[u](5,-2.5){$y$}
\psarc*[linecolor=gray](5,-4.5){1.5}{90}{270}
\psarc(5,-4.5){1.5}{0}{270} \uput[u](6.25,-3.25){$L_y(A)$}
\psline{->}(3.25,-4.5)(7,-4.5) \psline{->}(5,-6.25)(5,-2.5)
\end{pspicture}
\caption{Action of the hatching operators $L_x$ and $L_y$ on the set  $A$.}
\label{3f2}
\end{center}
\end{figure}

To give a more formal definition, let $A_x=\{p=(p_1,p_2)\in A\ |\ p_1=x\}$ and $A^y=\{p=(p_1,p_2)\in A\ |\ p_2=y\}$ for $A\subseteq\mathbb{R}_1^2$.
Given any subset $X\subseteq \mathbb{R}_1^2$ lying on a horizontal or vertical line in $\mathbb{R}_1^2$, denote the
minimal segment (possibly infinite or empty) containing the set $X$ and contained in the same horizontal or vertical line by $[X]$.
We can now define
\[
L_x(A)=\bigcup_{y\in\mathbb{R}}[A^y]
\]
and
\[
L_y(A)=\bigcup_{x\in\mathbb{R}}[A_x].
\]
We call the set $L_x(A)$ the hatching of $A$ in the $x$- direction and $L_y(A)$ the hatching of $A$ in the $y$- direction.
Note that $A\subseteq L_x(A)$ and $A\subseteq L_y(A)$.

We want to fix some notations for later use:

\begin{definition}
For $p=(p_1,p_2)\in\mathbb{R}_1^2$ and $\varepsilon_1,\varepsilon_2=\pm$, we call the set
\[
Q_p^{\varepsilon_1\varepsilon_2}=\{q=(q_1,q_2)\in \mathbb{R}_1^2\ |\ \varepsilon_i(q_i-p_i)\geq 0,\ i=1,2\}
\]
the $\varepsilon_1\varepsilon_2-$quadrant of $p$. For $A\subseteq\mathbb{R}_1^2$, $p\in\mathbb{R}_1^2$
and $\varepsilon_1,\varepsilon_2=\pm$, we denote the set $A\cap Q_p^{\varepsilon_1,\varepsilon_2}$ by $A_p^{\varepsilon_1,\varepsilon_2}$.

Furthermore we define for $p=(p_1,p_2)\in\mathbb{R}_1^2$
\[x_p=\{(p_1+t,p_2)\ |\ t\in\mathbb{R}\},\]
\[y_p=\{(p_1,p_2+t)\ |\ t\in\mathbb{R}\},\]
\[x_p^{\varepsilon}=\{(p_1+\varepsilon t,p_2)\ |\ t\geq0\},\]
\[y_p^{\varepsilon}=\{(p_1,p_2+\varepsilon t)\ |\ t\geq0\}.\]

We call the set $x_p^{\varepsilon}\cup y_p^{\delta}$ the $\varepsilon-\delta-$elbow of $p$.
\end{definition}

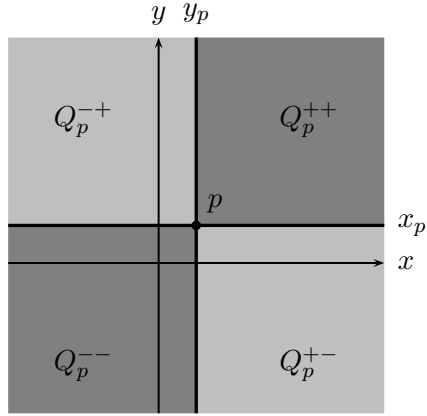
\begin{figure}[h]
\begin{center}
\begin{pspicture}(-2,-2)(3.25,3.25)
\pspolygon*[linecolor=gray](0.5,0.5)(0.5,3)(3,3)(3,0.5)
\pspolygon*[linecolor=lightgray](0.5,0.5)(0.5,3)(-2,3)(-2,0.5)
\pspolygon*[linecolor=lightgray](0.5,0.5)(0.5,-2)(3,-2)(3,0.5)
\pspolygon*[linecolor=gray](0.5,0.5)(0.5,-2)(-2,-2)(-2,0.5)
\uput[u](0,3){$y$} \psline[linewidth=0.75pt]{->}(0,-2)(0,3)
\psline[linewidth=0.75pt]{->}(-2,0)(3,0) \uput[r](3,0){$x$}
\psdot(0.5,0.5)
\psline[linewidth=1.25pt](0.5,-2)(0.5,3)
\psline[linewidth=1.25pt](-2,0.5)(3,0.5)
\uput[u](0.5,3){$y_p$}
\uput[r](3,0.5){$x_p$}
\uput[u](0.75,0.5){$p$}
\uput[u](2,1.5){$Q_p^{++}$}
\uput[u](-1,1.5){$Q_p^{-+}$}
\uput[u](2,-1.75){$Q_p^{+-}$}
\uput[u](-1,-1.75){$Q_p^{--}$}
\end{pspicture}
\caption{Quadrants of the point $p$.} \label{3f1}
\end{center}
\end{figure}

The operations $L_x$ and $L_y$ do not commute generally. For
example, for $A=\{(0,0),(2,0),(1,1)\}$, $(L_x\circ
L_y)(A)\neq(L_y\circ L_x)(A)$ (see Figure~\ref{3f3}). But for a
path connected set, these operations do commute as the following
proposition shows.

\begin{figure}[h]
\begin{center}
\begin{pspicture}(-0.5,-1.75)(10.5,3.5)
\psline[linewidth=0.5pt]{->}(0,-0.5)(0,1.75)
\psline[linewidth=0.5pt]{->}(-0.5,0)(2.5,0)
\uput[u](2.5,0){$x$} \uput[r](0,1.75){$y$}
\psdots(0,0)(1,1)(2,0)
\uput[u](2,1){$A$}
\psline[linewidth=0.5pt]{->}(4,1)(4,3.25)
\psline[linewidth=0.5pt]{->}(3.5,1.5)(6.5,1.5)
\psline[linewidth=1.5pt](4,1.5)(6,1.5)
\uput[u](6.5,1.5){$x$} \uput[r](4,3.25){$y$}
\psdots(4,1.5)(5,2.5)(6,1.5)
\uput[u](6,2.5){$L_x(A)$}
\psline[linewidth=0.5pt]{->}(8,1)(8,3.25)
\psline[linewidth=0.5pt]{->}(7.5,1.5)(10.5,1.5)
\psline[linewidth=1.5pt](8,1.5)(10,1.5)
\psline[linewidth=1.5pt](9,2.5)(9,1.5)
\uput[u](10.5,1.5){$x$} \uput[r](8,3.25){$y$}
\psdots(8,1.5)(9,2.5)(10,1.5)
\uput[u](10,2.5){$L_y(L_x(A))$}
\psline[linewidth=0.5pt]{->}(4,-2)(4,0.25)
\psline[linewidth=0.5pt]{->}(3.5,-1.5)(6.5,-1.5)
\uput[u](6.5,-1.5){$x$} \uput[r](4,0.25){$y$}
\psdots(4,-1.5)(5,-0.5)(6,-1.5)
\uput[u](6,-0.5){$L_y(A)$}
\psline[linewidth=0.5pt]{->}(8,-2)(8,0.25)
\psline[linewidth=0.5pt]{->}(7.5,-1.5)(10.5,-1.5)
\psline[linewidth=1.5pt](8,-1.5)(10,-1.5)
\uput[u](10.5,-1.5){$x$} \uput[r](8,0.25){$y$}
\psdots(8,-1.5)(9,-0.5)(10,-1.5)
\uput[u](10,-0.5){$L_x(L_y(A))$}
\end{pspicture}
\caption{The image of $A$ under $L_x$, $L_y$ and the different compositions of $L_x$ and $L_y$ yielding unequal sets.}
\label{3f3}
\end{center}
\end{figure}
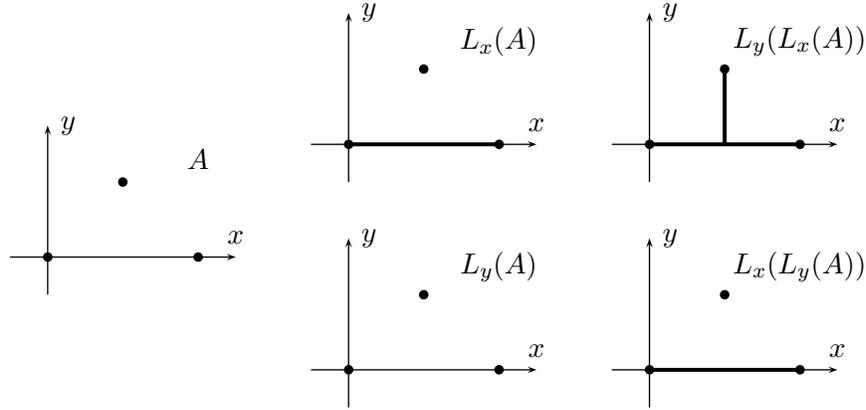

\begin{proposition}
\label{305}
For a path connected subset $A$ of $\mathbb{R}_1^2$, it holds
\[(L_x\circ L_y)(A)=(L_y\circ L_x)(A).\]
\end{proposition}
\begin{proof}
Assume to the contrary that, for example, there exists a point $p\in (L_y\circ L_x)(A)$ not contained in
$(L_x\circ L_y)(A)$. Since $p\notin L_x (L_y(A))$, the point $p$ can not belong to $L_y(A)$ and consequently
at least one of the half-rays $y_p^+$ and $y_p^-$ does not intersect $A$ (otherwise the point $p$
would belong to $L_y(A)$). Now assume without loss of generality that $y_p^+$ does not intersect $A$.
Since $p\in L_y(L_x(A))$, there must exist points $u,v\in L_x(A)$ with $u\in y_p^+$ and $v\in y_p^-$. Then,
$u\in L_x(A)$ must lie on a certain horizontal segment $[A^y]$. Choose two points $s,t\in A$ on this segment lying
in $Q_p^{-+}$ and $Q_p^{++}$ respectively (see Figure~\ref{3f4}). Since $A$ is path connected, there is a path in $A$ connecting
these two points. As $y_p^+\cap A=\emptyset$, this path has to intersect the half-rays $x_p^-$ and $x_p^+$.
If we denote the intersection points by $q$ and $r$, then $q,r\in A\subseteq L_y(A)$ and this implies
$p\in L_x(L_y(A))$, contradicting our assumption.
\end{proof}

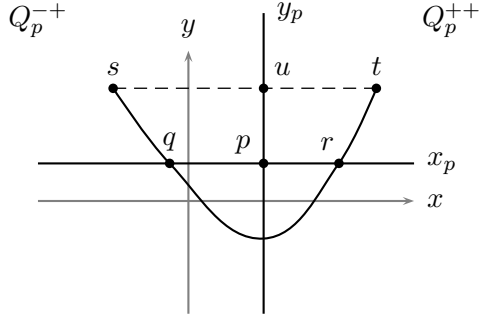
\begin{figure}[h]
\begin{center}
\begin{pspicture}(-2,-1.5)(3.5,2.5)
\psline[linecolor=gray]{->}(-2,0)(3,0)
\psline[linecolor=gray]{->}(0,-1.5)(0,2) \uput[r](3,0){$x$}
\uput[u](0,2){$y$} \psdots(1,0.5)(-1,1.5)(2.5,1.5)(1,1.5)
\uput[u](-1,1.5){$s$}
\uput[u](2.5,1.5){$t$}
\uput[u](1.25,1.5){$u$}
\uput[l](1,0.75){$p$}
\psline[linestyle=dashed,linewidth=0.5pt](-1,1.5)(2.5,1.5)
\psline(-2,0.5)(3,0.5)
\psline(1,-1.5)(1,2.5) \uput[r](3,0.5){$x_p$}
\uput[r](1,2.5){$y_p$}
\pscurve(-1,1.5)(-0.25,0.5)(1,-0.5)(2,0.5)(2.5,1.5)
\psdots(-0.25,0.5)(2,0.5) \uput[u](-0.25,0.5){$q$}
\uput[u](1.85,0.5){$r$}
\uput[u](-2,2){$Q_p^{-+}$}
\uput[u](3.5,2){$Q_p^{++}$}
\end{pspicture}
\caption{Why $L_x$ and $L_y$ commute for path connected subsets.}
\label{3f4}
\end{center}
\end{figure}

We can now define the ``double hatching'' of connected subsets of the Manhattan plane:
\begin{definition}
For a path connected subset $A\subseteq \mathbb{R}_1^2$, we define $L(A)=(L_x\circ L_y)(A)=(L_y\circ L_x)(A)$ and call it the double hatching of $A$.
\end{definition}

\begin{lemma}
\label{31}
Let $A\subseteq\mathbb{R}_1^2$ be a path connected subset and $p\in \mathbb{R}_1^2$. If every quadrant of the point $p$
contains a point of the set $A$, then $p\in L(A)$.
\end{lemma}
\begin{proof}
Let $a_1\in A_p^{++}$, $a_2\in A_p^{-+}$, $a_3\in A_p^{--}$, $a_4\in A_p^{+-}$. Since $A$ is path connected,
there exists a path $\alpha$ connecting $a_1$ and $a_3$; and a path $\beta$ connecting $a_2$ and $a_4$ (see Figure~\ref{3f5}). The path
$\alpha$ intersects either the pair of half-rays $x_p^-$ and $y_p^+$ or the pair of half-rays $x_p^+$ and $y_p^-$ (or both pairs).
Similarly, the path $\beta$ intersects either the pair of half-rays $x_p^+$ and $y_p^+$ or the pair of half-rays
$x_p^-$ and $y_p^-$. Without loss of generality, let us assume that the path $\alpha$ intersects $x_p^-$ and $y_p^+$
and the path $\beta$ intersects $x_p^+$ and $y_p^+$. Denote a point of intersection of $\alpha$ and $x_p^-$ by $q$
and a point of intersection of $\beta$ and $x_p^+$ by $r$. Since $q,r\in A$ the segment $[q,r]$ is contained in
$L_x(A)$, so that $p\in L_x(A)\subseteq L(A)$.

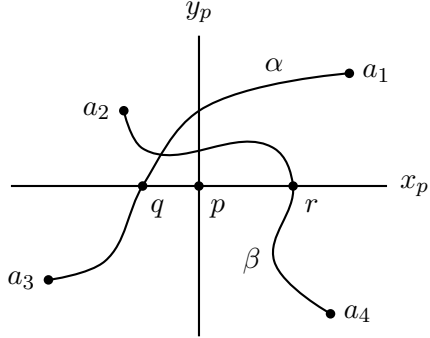
\begin{figure}[h]
\begin{center}
\begin{pspicture}(-2.5,-2)(3,2.5)
\psline(-2.5,0)(2.5,0)
\psline(0,-2)(0,2)
\psdots(0,0)(2,1.5)(1.75,-1.7)(-2,-1.25)(-1,1) \uput[d](0.25,0){$p$}
\uput[r](2,1.5){$a_1$} \uput[u](1,1.35){$\alpha$}
\uput[l](1,-1){$\beta$} \uput[r](1.75,-1.7){$a_4$}
\uput[l](-2,-1.25){$a_3$} \uput[l](-1,1){$a_2$}
\uput[r](2.5,0){$x_p$} \uput[u](0,2){$y_p$}
\pscurve(2,1.5)(0,1)(-0.75,0)(-1.25,-1)(-2,-1.25)
\pscurve(-1,1)(-0.75,0.5)(1,0.5)(1.25,0)(1,-1)(1.75,-1.7)
\uput[d](-0.55,0){$q$} \uput[d](1.5,0){$r$} \psdots(1.25,0)(-0.75,0)
\end{pspicture}
\caption{A point, whose quadrants intersect a set, is in the double hatching of the set.}
\label{3f5}
\end{center}
\end{figure}
\end{proof}

\begin{lemma}
\label{32}
Let $A\subseteq\mathbb{R}_1^2$ be a path connected subset, $p\in\mathbb{R}_1^2$ and $\varepsilon,\delta=\pm$.
Then $A_p^{\varepsilon\delta}$ is nonempty if and only if $(L(A))_p^{\varepsilon\delta}$ is nonempty.
\end{lemma}
(We omit the easy proof).

From Lemma~\ref{31} and Lemma~\ref{32} we get the following corollary:

\begin{corollary}
Let $A\subseteq\mathbb{R}_1^2$ be a path connected subset and $p\in \mathbb{R}_1^2$. If every quadrant
of the point $p$ intersects the set $L(A)$, then $p\in L(A)$.
\end{corollary}

\begin{definition}
A path connected subset $A\subseteq\mathbb{R}_1^2$ is called $L-$invariant if the property $L(A)=A$ holds.
\end{definition}

By the previous corollary, we get the following lemma:

\begin{lemma}
Let $A\subseteq\mathbb{R}_1^2$ be an $L-$invariant subset and $p\in\mathbb{R}_1^2$. If every quadrant
of the point $p$ has nonempty intersection with the set $A$, then $p$ belongs to $A$.
\end{lemma}

\begin{proposition}
\label{33}
For every path connected subset $A\subseteq\mathbb{R}_1^2$, the set $L(A)$ is $L-$invariant.
\end{proposition}
\begin{proof}
It can easily be seen that $L_x\circ L_x=L_x$ and $L_y\circ
L_y=L_y$ on any subset of $\mathbb{R}_1^2$. On the other side,
$L_x$ and $L_y$ commute on path connected subsets by
Proposition~\ref{305}. So, for path connected subsets, we get

\begin{eqnarray*}
L\circ L&=&(L_x\circ L_y)\circ(L_x\circ L_y)=(L_x\circ
L_y)\circ(L_y\circ L_x)\\ &=& L_x\circ L_y\circ L_x= L_x\circ
L_x\circ L_y=L_x\circ L_y\\&=& L
\end{eqnarray*}

Since $A$ is path connected, $L(A)$ is also path connected, and
thus $L-$invariant.
\end{proof}

\section{Constructing the Tight Span of Path Connected Subsets}

We will show below (Theorem~\ref{35}) that the closure of the
double-hatching of any path connected subset of the Manhattan
plane is isometric to the tight span of this subset. We will prove
this via the hyperconvexity of the closure of any $L$-invariant
subset (Theorem~\ref{34}). We will first note some facts used in
the following proofs:

\begin{proposition}
\label{32.1} (\cite{esp}) Any hyperconvex metric space is
complete.
\end{proposition}

\begin{proposition}
\label{32.2}
(\cite{kil})
Any hyperconvex metric space is strictly intrinsic.
\end{proposition}

\begin{lemma}
\label{at5} Let $(X,d_X)$ and $(Y,d_Y)$ be metric spaces with
$X\cap Y\neq\emptyset$, and assume $d_X|_{X\cap Y}=d_Y|_{X\cap
Y}$. Then there exists  a metric $d$ on $X\cup Y$, such that
$d|_X=d_X$ and $d|_Y=d_Y$.
\end{lemma}

\begin{theorem}
\label{34}
Let $A\subseteq\mathbb{R}_1^2$ be an $L-$invariant subset. Then $\overline{A}$ is hyperconvex.
\end{theorem}
\begin{proof}
To show that $\overline{A}$ is hyperconvex, we have to find a nonexpansive retraction $r_q:\overline{A}\cup\{q\}\rightarrow \overline{A}$ where
$\overline{A}\cup\{q\}$ is any one-point extension of the metric space $\overline{A}$. We can however assume that the point $q$ belongs to
$\mathbb{R}_1^2\setminus \overline{A}$ and $\overline{A}\cup\{q\}$ carries the induced metric from the Manhattan plane. The reason for this
simplification is that any metric on $\overline{A}\cup\{q\}$ can be extended to $\mathbb{R}_1^2\cup\{q\}$ by Lemma~\ref{at5} and there exists
a nonexpansive retraction $r:\mathbb{R}_1^2\cup\{q\}\rightarrow \mathbb{R}_1^2$ since $\mathbb{R}_1^2$ is injective. If we now
denote the point $r(q)\in \mathbb{R}_1^2$ by $p$, any nonexpansive retraction $\overline{A}\cup\{p\}\rightarrow \overline{A}$ can be combined with
the retraction $r$ to yield a nonexpansive retraction $\overline{A}\cup\{q\}\rightarrow \overline{A}$.

So, let us given any point $p=(p_1,p_2)\in\mathbb{R}_1^2\setminus\overline{A}$. We shall construct a nonexpansive retraction
$\overline{A}\cup\{p\}\rightarrow\overline{A}$ by considering three different cases:

1) Three Quadrants Case:

Assume the set $A$ has nonempty intersections with exactly three quadrants of the point $p$. Without loss of generality, we assume the quadrants
to be $Q_p^{++}$, $Q_p^{+-}$ and $Q_p^{-+}$.

Since $p\in \mathbb{R}_1^2\setminus\overline{A}$, there exists an interval $[0,\varepsilon)$ such that the elbows $x_{p+(1,1)t}^-\cup y_{p+(1,1)t}^-$
of the points $p+(1,1)t$ for $t\in[0,\varepsilon)$ does not intersect $A$. Otherwise we would have points $p+(1,1)t$, for $t$ small enough
such that four quadrants of the points $p+(1,1)t$ would intersect $A$ forcing these points to belong to the set $A$. So the point $p$ itself
would belong to $\overline{A}$, contrary to the assumption.

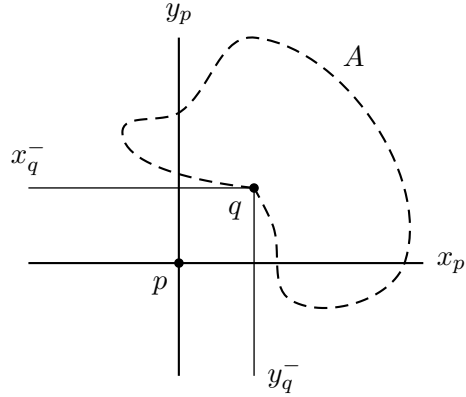
\begin{figure}[h]
\begin{center}
\begin{pspicture*}(-2.25,-1.75)(4,3.5)
\psline(0,-1.5)(0,3) \psline(-2,0)(3.25,0)
\psline[linewidth=0.5pt](-2,1)(1,1)(1,-1.5) \psdots(0,0)(1,1)
\uput[d](-0.25,0){$p$} \uput[d](0.75,1){$q$}
\pscurve[linestyle=dashed](1,1)(-0.75,1.75)(0,2)(1,3)(3,0)(1.5,-0.5)(1.25,0.5)(1,1)
\uput[u](0,3){$y_p$} \uput[r](3.25,0){$x_p$}
\uput[r](1,-1.5){$y_q^-$} \uput[u](-2,1){$x_q^-$}
\uput[r](2,2.75){$A$}
\end{pspicture*}
\caption{The three quadrants case.}
\label{3f6}
\end{center}
\end{figure}

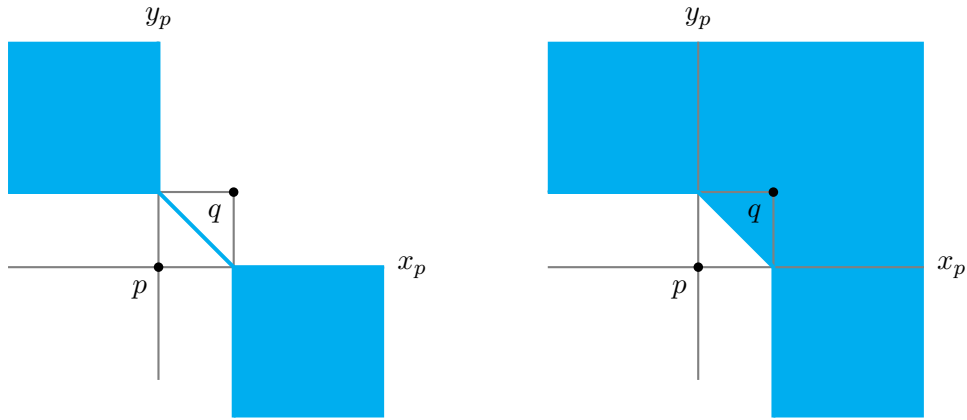
\begin{figure}[h!]
\begin{minipage}[t]{0.47\linewidth}
\centering
\begin{pspicture*}(-2.15,-2)(4,3.5)
\psline[linecolor=gray](0,-1.5)(0,3) \psline[linecolor=gray](-2,0)(3,0)
\psline[linecolor=gray](-2,1)(1,1)(1,-1.5) \psdots(0,0)(1,1)
\pspolygon*[linecolor=cyan,fillcolor=cyan](0,3)(0,1)(-2,1)(-2,3)(0,3)
\psline[linewidth=1.5pt, linecolor=cyan](0,3)(0,1)(-2,1)
\psline[linewidth=1.5pt, linecolor=cyan](0,1)(1,0)
\pspolygon*[linecolor=cyan](3,0)(1,0)(1,-2)(3,-2)(3,0)
\psline[linewidth=1.5pt, linecolor=cyan](3,0)(1,0)(1,-2)
\uput[d](-0.25,0){$p$} \uput[d](0.75,1){$q$}
\uput[u](0,3){$y_p$} \uput[r](3,0){$x_p$}
\end{pspicture*}

\end{minipage}
\begin{minipage}[t]{0.47\linewidth}
\centering
\begin{pspicture*}(-3.25,-2)(4,3.5)
\pspolygon*[linecolor=cyan](-2,1)(0,1)(1,0)(1,-2)(3,-2)(3,3)(-2,3)(-2,1)
\psline[linecolor=gray](0,-1.5)(0,3) \psline[linecolor=gray](-2,0)(3,0)
\psline[linecolor=gray](-2,1)(1,1)(1,-1.5) \psdots(0,0)(1,1)
\psline[linewidth=1.5pt, linecolor=cyan](-2,1)(0,1)(1,0)(1,-2)
\uput[d](-0.25,0){$p$} \uput[d](0.75,1){$q$}
\uput[u](0,3){$y_p$} \uput[r](3,0){$x_p$}
\end{pspicture*}

\end{minipage}

\caption{The shaded (blue) region denotes the set of points with equal distance to the points $p$ and $q$. Note however that the ``true'' middle points (i.e. with distance $\frac{1}{2}\cdot d_1(p,q)$ to $p$ and $q$) are the points on the anti-diagonal segment connecting the two quarter planes (left). The shaded region denotes the set of points whose distance to $q$ is no greater than the distance to $p$ (right).}
\label{3f7}

\end{figure}

Now, if we take the supremum of such $\varepsilon$, say
$\varepsilon_0$, then the point $q=p+(1,1)\varepsilon_0$ (see
Figure~\ref{3f6}) must belong to the set $\overline{A}$ by the
same reason.

Now we can define a nonexpansive retraction
\[r:\overline{A}\cup\{p\}\rightarrow \overline{A},\]
\[
r(x)=
\left\{
\begin{array}{rrr}
x&,&\mbox{for } x\in \overline{A}\\
q&,&\mbox{for } x=p \ ,\\
\end{array}
\right.
\]
as easily can be verified by simple geometry in taxicab metric
(see Figure~\ref{3f7}).

2) Two Quadrants Case:

Assume that the set $A$ intersects exactly two quadrants. Since $p\notin A$, these quadrants can not be in diagonal position
by the connectedness of $A$. So let us assume without loss of generality that these quadrants are $Q_p^{++}$ and $Q_p^{-+}$.

Let us define
\[
\inf\{t>0|\ x_{(p_1,p_2+t)}^+\cap A\neq \varnothing\}=t_0
\]
and
\[
\inf\{t>0|\ x_{(p_1,p_2+t)}^-\cap A\neq \varnothing\}=t_1
\]

Without loss of generality we can assume to $t_0\leq t_1$.
Consider the point $q=(p_1,p_2+t_0)$ (see Figure~\ref{3f9}) and
define the map
\[r:\overline{A}\cup\{p\}\rightarrow\overline{A}\cup\{q\}, \]
\[
r(x)=
\left\{
\begin{array}{rrr}
x&,&\mbox{for } x\in \overline{A}\\
q&,&\mbox{for } x=p\ .\\
\end{array}
\right.
\]

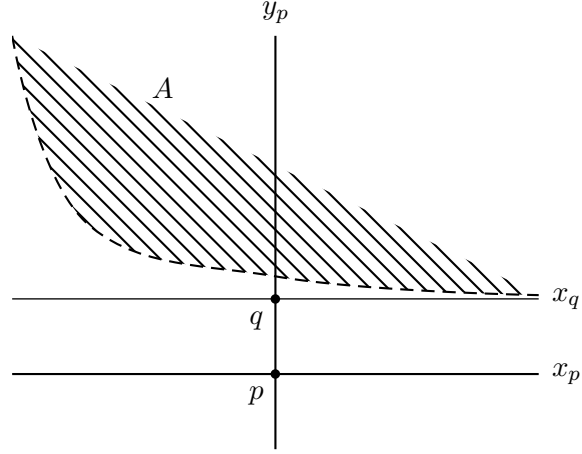
\begin{figure}[h]
\begin{center}
\begin{pspicture*}(-3.5,-1)(4.25,5)
\psline(-3.5,0)(3.5,0) \psline(0,-1)(0,4.5) \uput[r](3.5,0){$x_p$}
\uput[u](0,4.5){$y_p$} \psline[linewidth=0.5pt](-3.5,1)(3.5,1)
\uput[r](3.5,1){$x_q$}
\pscurve[linestyle=dashed,fillstyle=vlines,fillcolor=gray](-3.5,4.5)(-2.5,2)(0,1.3)(3.5,1.05)
\uput[u](-1.5,3.5){$A$} \uput[d](-0.25,0){$p$}
\uput[d](-0.25,1){$q$} \psdots(0,0)(0,1)
\end{pspicture*}
\caption{The two quadrants case.}
\label{3f9}
\end{center}
\end{figure}

One can easily verify that this map is a nonexpansive map. If $q$ belongs to the set $\overline{A}$, then we are done.

Now consider the case $q\notin \overline{A}$. We first claim that in this case it must hold $t_0<t_1$. Because otherwise any horizontal
segment with ordinate close enough to $p_2+t_0$ from above would intersect the set $A$ producing a point just above the point $q$
and so $q$ would belong to $\overline{A}$.

Since $\overline{A}$ is closed, there exists an open disc
$B(q,\delta)$ disjoint from $\overline{A}$. If we now choose
$\varepsilon<\min\{\delta,t_1-t_0\}$ then exactly three quadrants
of the point $q'=(p_1+\varepsilon,p_2+t_0+\varepsilon)$ intersects
the set $A$ and all points of $\overline{A}$ are closer to the
point $q'$ than to the point $q$. So we can first define a
nonexpansive map $r':\overline{A}\cup\{q\}\rightarrow
\overline{A}\cup\{q'\}$ and then, by applying the three quadrant
case construct a nonexpansive retraction
$r'':\overline{A}\cup\{q'\}\rightarrow \overline{A}$. By combining
these three retractions we get a nonexpansive retraction from
$\overline{A}\cup\{p\}$ to $\overline{A}$.

3) One Quadrant Case:

Assume that only one quadrant of the point $p$ intersects the set
$A$ and let it be $Q_p^{++}$ without loss of generality. The elbow
$x_p^+\cup y_p^+$ of the point $p$ can not intersect the set $A$.
Now consider intervals $[0,\varepsilon)$ such that the
corresponding elbows of the points $p+(1,1)t$ does not intersect
the set $A$ for all $t\in [0,\varepsilon)$. Let $\varepsilon_0$
denote the supremum of such $\varepsilon$ (which might be zero in
special examples) and let $q=p+(1,1)\varepsilon_0$ (see
Figure~\ref{3f10}). The map

\begin{figure}[h]
\begin{center}
\begin{pspicture*}(-1,-1)(5.25,5.25)
\psline(-1,0)(4.5,0) \psline(0,-1)(0,4.5) \uput[r](4.5,0){$x_p$}
\uput[u](0,4.5){$y_p$} \psline[linewidth=0.5pt](4.5,1)(1,1)(1,4.5)
\uput[r](4.5,1){$x_q^+$} \uput[u](1,4.5){$y_q^+$}
\pscurve[linestyle=dashed,fillstyle=vlines,fillcolor=gray](1.075,4.5)(1.75,2)(2.5,1.5)(3,3)(4.5,4.5)
\uput[r](3.5,3.5){$A$} \uput[d](-0.25,0){$p$}
\uput[d](0.75,1){$q$} \psdots(0,0)(1,1)
\end{pspicture*}
\caption{The one quadrant case.}
\label{3f10}
\end{center}
\end{figure}
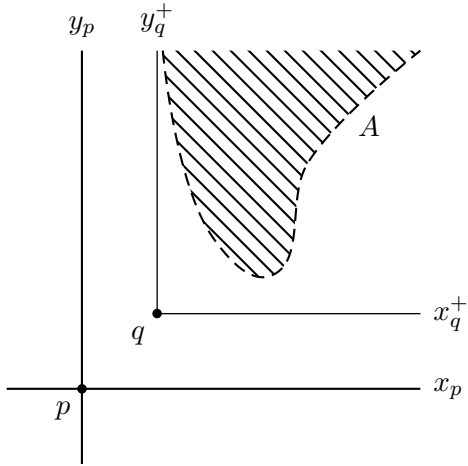

\[r:\overline{A}\cup\{p\}\rightarrow\overline{A}\cup\{q\} , \]
\[
r(x)=
\left\{
\begin{array}{rrr}
x&,&\mbox{for } x\in \overline{A}\\
q&,&\mbox{for } x=p\\
\end{array}
\right.
\]
is nonexpansive.

If $q\in \overline{A}$, then we are done, as the map $r$ is then a
nonexpansive retraction on to $\overline{A}$. Now assume
$q\notin\overline{A}$. If we move the point $q$ to a nearby
$q'=q+(1,1)\delta$ for small $\delta$, then at least two and at
most three quadrants of the point $q'$ intersect the set $A$ and
we can apply the previous cases.
\end{proof}

\begin{theorem}
\label{35}
Let $A\subseteq\mathbb{R}_1^2$ be a path connected subset. Then, $\overline{L(A)}$ is isometric to the tight span $T(A)$ of $A$.
\end{theorem}
\begin{proof}
A geodesically convex subset $B\subseteq\mathbb{R}_1^2$ containing
the set $A$ has to contain also the double-hatching $L(A)$ of $A$,
since during the process of horizontal or vertical hatching the
unique geodesics between two horizontally or vertically positioned
points of $A$ are added to the set $A$. As $\mathbb{R}_1^2$ is
hyperconvex, it contains an isometric copy of the tight span
$T(A)$ and since $T(A)$ is strictly intrinsic (by
Proposition~\ref{32.2}), and thus geodesically convex in
$\mathbb{R}_1^2$, it has to contain $L(A)$: $L(A)\subseteq T(A)$.
Consequently $T(A)=T(L(A))$. By Proposition~\ref{33} and
Theorem~\ref{34}, $\overline{L(A)}$ is hyperconvex. Now by
Proposition~\ref{32.1}, $T(L(A))=\overline{L(A)}$. Thus we get
$T(A)=\overline{L(A)}$.
\end{proof}

\begin{corollary}\label{cor}
Let $A\subseteq\mathbb{R}_1^2$, be a compact and path connected subset. Then, $L(A)$ is isometric to the tight span $T(A)$ of $A$.
\end{corollary}
\begin{proof}
It can easily be shown that for a compact $A$, $L(A)$ will also be
compact.
\end{proof}

For an example of Corollary \ref{cor} see Figure~\ref{vh}.

\begin{figure}[h]
\includegraphics[width=0.9\textwidth]{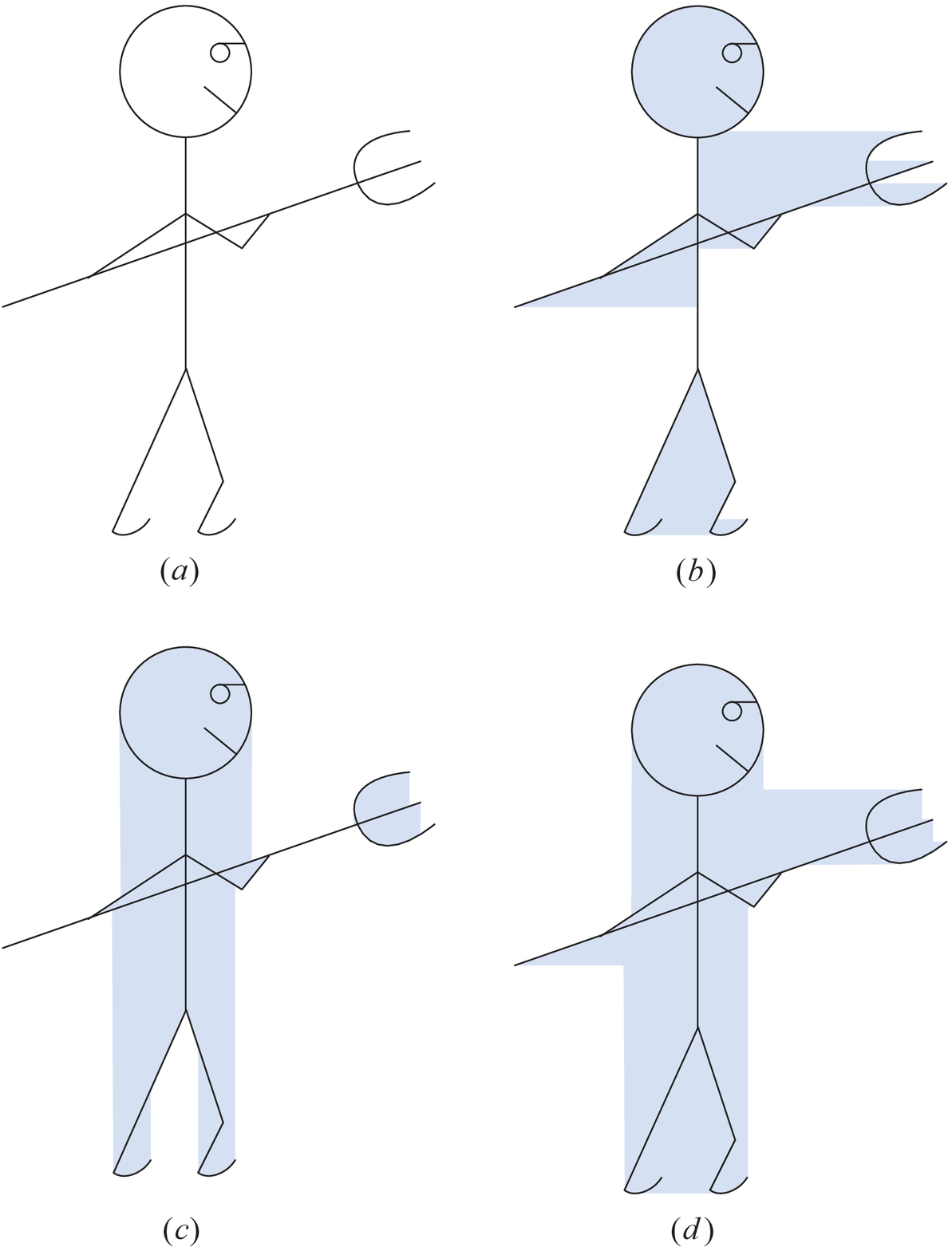}
\caption{Tight span of a compact and path connected subset $A$.
\newline a) A compact and path connected subset $A$ of $\mathbb{R}_1^2$ (consisting of the contour).
\newline b) The horizontal hatching $L_x(A)$ of $A$.
\newline c) The vertical hatching $L_y(A)$ of $A$.
\newline d) The double hatching $L_x(L_y(A))=L_y(L_x(A))$ of $A$, which is isometric to the tight span $T(A)$ of
$A$.}\label{vh}
\end{figure}

\end{document}